\documentclass[12pt]{amsart}
\usepackage{amsthm}
\usepackage{amssymb}
\usepackage{amsmath}
\usepackage{mathrsfs}
\usepackage{graphicx, enumerate}
\usepackage{color, soul}
\usepackage[colorlinks=true,linkcolor=red,citecolor=blue]{hyperref}

\def\A{\mathcal{A}}
\def\M{\mathcal {M}}
\def\H{\mathcal {H}}
\def\G{\mathcal G}
\def\D{\mathbb D}
\def\F{\mathscr{F}}

\def\v{\overset\vee}

\newtheorem{theorem}{Theorem}[section]
\newtheorem{lemma}[theorem]{Lemma}

\newtheorem{proposition}[theorem]{Proposition}
\newtheorem{remark}[theorem]{Remark}
\newtheorem{example}[theorem]{Example}

\begin{document}

\title{A fibered description of the vector-valued spectrum}

\author[V. Dimant]{Ver\'onica Dimant}

\author[J. Singer]{Joaqu\'{\i}n Singer}

\thanks{Partially supported by Conicet PIP 11220130100483  and ANPCyT PICT 2015-2299 }


\address{Departamento de Matem\'{a}tica y Ciencias, Universidad de San
	Andr\'{e}s, Vito Dumas 284, (B1644BID) Victoria, Buenos Aires,
	Argentina and CONICET} \email{vero@udesa.edu.ar}

\address{Departamento de Matem\'{a}tica, Facultad de Ciencias Exactas y Naturales, Universidad de Buenos Aires, (1428) Buenos Aires,
	Argentina and IMAS-CONICET} \email{jsinger@dm.uba.ar}

\begin{abstract}
For Banach spaces $X$ and $Y$ we study the vector-valued spectrum $\M_\infty(B_X,B_Y)$, that is the set of non null algebra homomorphisms from $\H^\infty(B_X)$ to $\H^\infty(B_Y)$, which is naturally projected onto the closed unit ball of $\H^\infty(B_Y, X^{**})$. The aim of this article is to describe the fibers defined by this projection, searching for analytic balls and considering Gleason parts.
\end{abstract}

\maketitle

\section{Introduction}
Let  $\H^\infty(B_X)$ be the space of bounded holomorphic functions on the open unit ball of a complex Banach space $X$. The study of the spectrum of this uniform algebra $\M(\H^\infty(B_X))$ began with the seminal work of Aron, Cole and Gamelin \cite{AronColeGamelin} where this set was fibered over $\overline B_{X^{**}}$, the closed unit ball of the bidual of $X$. For $X$  infinite dimensional, unlike what happens in the one dimensional case,  it was proved in the same article that each fiber is quite large. Following this way the description of the fibers and the study of conditions that assure the existence of analytic balls inside the fibers was addressed in several articles, as \cite{ColeGamelinJohnson, Farmer, AronFalcoGarciaMaestre}.

Inspired by what is known about the spectrum $\M(\H^\infty(B_X))$ our aim here is to study, for Banach spaces $X$ and $Y$, the \textit{vector-valued spectrum} $\M_\infty(B_X,B_Y)$ defined by
\[
\M_\infty(B_X,B_Y)=\{ \Phi: \H^\infty(B_X)\to \H^\infty(B_Y) \textrm{ algebra homomorphisms}\}\setminus \{0\}.
\]

Even if  homomorphisms between uniform algebras are a typical object of study (see, for instance, \cite{Gamelin, GalindoGamelinLindstrom, GalindoLindstrom, GorkinMortini, AronGalindoLindstrom,MacCluerOhnoZhao}), the treatment of this set as a whole just started in \cite{DiGaMaSe}. Now, we continue that work with a slight change of perspective and focus but maintaining a  structure modeled on the ideas of \cite{AronColeGamelin}.

As was noticed in \cite{AronColeGamelin}, in order to obtain information about the spectrum of $\H^\infty(B_X)$ it is useful  to first study the spectrum of $\H_b(X)$, the Fr\'{e}chet algebra of holomorphic functions of bounded type on $X$ (that is, holomorphic functions which are bounded on bounded sets, with the topology of uniform convergence on bounded sets). The same idea leads our work here: with the goal of describing the vector-valued spectrum  $\M_\infty(B_X,B_Y)$ we begin by focusing on the set $\M_{b,\infty}(X,B_Y)$ given by
\[
\M_{b,\infty}(X,B_Y)=\{ \Phi: \H_b(X)\to \H^\infty(B_Y) \textrm{ continuous algebra homomorphisms}\}\setminus \{0\}.
\]

As in the scalar-valued case, this set has a rich analytic structure in the form of a Riemann domain over the space  $\mathcal{H}^{\infty}(B_Y,X^{**}) $, when $X$ is symmetrically regular (see  Section \ref{Section-Riemann domain}). The study of the fibers of $\M_{b,\infty}(X,B_Y)$ over $\mathcal{H}^{\infty}(B_Y,X^{**}) $ is developed in Section \ref{Section-Fiber Mb}. Following the Aron-Cole-Gamelin program, a radius function can be defined for the homomorphisms in $\M_{b,\infty}(X,B_Y)$ and then extended to $\M_\infty(B_X,B_Y)$ giving a way to relate both spectra. This is presented in  Section \ref{Section-Radius function}.

Each function $g\in \mathcal{H}^{\infty}(B_Y,X^{**}) $ with $g(B_Y)\subset B_{X^{**}}$ naturally produces a composition homomorphism $ C_g\in \M_\infty(B_X,B_Y)$ given by
\[
 C_g(f)=\widetilde f\circ g,\quad\textrm{ for all }f\in\H^\infty(B_X),
\] where $\widetilde f\in \H^\infty(B_{X^{**}})$ is the canonical extension  of $f$ (see reference below). Conversely, as in \cite{DiGaMaSe}, we can define a projection
\begin{align*}
			\xi \colon \M_\infty(B_X,B_Y) &\to \mathcal{H}^\infty(B_Y,X^{**}), \\
			\Phi & \mapsto \left[ y \mapsto [x^* \mapsto \Phi(x^*)(y)] \right].
		\end{align*}
The image of this projection is the closed unit ball of $\mathcal{H}^\infty(B_Y,X^{**})$ (see explanation in Section \ref{Section-Minf}).
One of the goals of this article is to describe the fibers over this closed ball, that is, for each $g\in \overline B_{\H^\infty(B_Y,X^{**})} $, the set of homomorphisms $\Phi\in\M_\infty(B_X,B_Y)$ such that $\xi(\Phi)=g$.

Finally, Section \ref{Section-Gleason parts} looks into the notion of Gleason parts for the vector-valued spectrum $\M_\infty(B_X,B_Y)$.

We begin by recalling some usual definitions and properties about polynomials and holomorphic functions in Banach spaces. For general theory on the topic we refer the reader to the books of Dineen \cite{Dineen}, Mujica \cite{MujicaLibro} and Chae \cite{Chae}.

Given Banach spaces $X$ and $Y$ we say that a function $P:X\to Y$ is a {\em continuous $m$-homogeneous polynomial} if there exists a unique continuous symmetric $m$-linear mapping $\v{P}$ such that $P(x) = \v{P}(x,\dots, x)$. If $U\subset X$ is an open set, a mapping $f:U\to Y$ is said to be {\em holomorphic} if for every $x_0\in U$ there exists a sequence $(P_mf(x_0))$, with each $P_mf(x_0)$ a continuous $m$-homogeneous polynomial, such that the series
\[
f(x)=\sum_{m=0}^\infty P_mf(x_0)(x-x_0)
\] converges uniformly in some neighborhood of $x_0$ contained in $U$.

We say that an $m$-homogeneous polynomial $P:X\to\mathbb C$ is {\em of finite type} if there are linear forms $x_1^*, \dots, x_n^*$ in $X^*$ such that
\[
P(x)=\sum_{k=1}^n (x_k^*(x))^m.
\]

The set
\[
\H^\infty(B_X)=\{f:B_X\to\mathbb C:\, f \textrm{ is holomorphic and bounded}\}
\] is a Banach algebra (endowed with the supremum norm). Analogously, the notation $\H^\infty(B_Y,X^{**})$ refers to the Banach space of bounded holomorphic functions from $B_Y$ to $X^{**}$ (endowed with the supremum norm).

A holomorphic function $f:X\to \mathbb C$ is said to be {\em of bounded type} if it maps bounded subsets of $X$ into bounded subsets of $\mathbb C$. The set
\[
\H_b(X)=\{f:X\to\mathbb C:\, f \textrm{ is a holomorphic function of bounded type}\}
\] is a Fr\'{e}chet algebra if we endow it with the family of (semi)norms $\{\sup_{\|x\|<R}|f(x)|\}_{R>0}$.

By \cite{AronBerner, DavieGamelin}, there is a canonical extension $[f\mapsto \widetilde f]$  from $\H^\infty(B_X)$ to $\H^\infty(B_{X^{**}})$ which is an isometry and  a homomorphism of Banach algebras. The extension is also defined from $\H_b(X)$ to $\H_b(X^{**})$  and satisfies, for any $f\in\H_b(X)$ and  $R>0$,
\[
\|f\|_{RB_X} = \|\widetilde f\|_{RB_{X^{**}}}.
\]

Recall that a Banach space $X$ is said to be {\em symmetrically regular} if  every continuous  linear mapping $T:X\to X^*$ which is symmetric (i. e. $T(x_1)(x_2)=T(x_2)(x_1)$ for all $x_1, x_2\in X$) turns out to be weakly compact.

The (scalar-valued) spectrum of a Banach or Fr\'{e}chet algebra $\mathcal A$ is the set
\[
\M(\A)=\{\varphi:\A\to\mathbb C \textrm{ algebra homomorphisms}\}\setminus \{0\}.
\]

As in \cite{AronColeGamelin}, we will denote the scalar-valued spectrum of $\H_b(X)$ by $\M_b(X)$. Analogously, $\M_\infty(B_X)$ will denote $\M(\H^\infty(B_X))$.

\textbf{Duality and compactness.} We quote some observations about duality and compactness for future reference.
\begin{enumerate}[(1)]
	\item As mentioned in \cite{GGL}, the algebra $\H^\infty (B_Y)$ is a weak-star closed subalgebra of $\ell_\infty(B_Y)$ and hence it is the dual of the quotient Banach space $\ell_1(B_Y)/\H^\infty(B_Y)^\perp$. We chose to denote this predual space by $\mathcal G^\infty (B_Y)$, as in \cite{Mujica}, for simplicity but we recall its description as a quotient of $\ell_1(B_Y)$.
	\item \label{compact} The vector-valued spectrum $\M_\infty(B_X,B_Y)$ is a subset of the unit sphere of
	\[\mathcal L(\H^\infty(B_X), \H^\infty(B_Y)) = \left(\H^\infty(B_X)\widehat\otimes_\pi \mathcal G^\infty (B_Y)\right)^*.\] As the closed unit ball of a dual Banach space is weak-star compact and $\M_\infty(B_X,B_Y)$ is closed with respect to this topology, it follows that $\M_\infty(B_X,B_Y)$ is a weak-star compact set (i. e. a compact set with respect to the topology given by $\H^\infty(B_X)\widehat\otimes_\pi \mathcal G^\infty (B_Y)$). Note also that, by the definition of $\mathcal G^\infty (B_Y)$,  a bounded net $(\Phi_\alpha)$  is weak-star convergent to $\Phi$ in $\left(\H^\infty(B_X)\widehat\otimes_\pi \mathcal G^\infty (B_Y)\right)^*$ if and only if $\Phi_\alpha (f)(y)\to \Phi(f)(y)$, for all $f\in \H^\infty(B_X)$ and $y\in B_Y$. The weak-star compactness of $\M_\infty(B_X,B_Y)$ is  shown with a different argument in \cite[Theorem 11]{DiGaMaSe}.
	\item \label{comment3} For each $R>0$, the norm $\|f\|_R=\sup\{|f(x)|:\ x\in RB_X\}$ gives to $\H_b(X)$ the structure of a normed space (not complete). In this way the set
	\[
	\M_{b,\infty}(X,B_Y)_R =\{\Phi\in \M_{b,\infty}(X,B_Y):\ \Phi\textrm{ is continuous with respect to the norm } \|\cdot\|_R\}
	\]
	is contained in the unit sphere of
	\[\mathcal L((\H_b(X),\|\cdot\|_R), \H^\infty(B_Y)) = \left((\H_b(X),\|\cdot\|_R)\otimes_\pi \mathcal G^\infty (B_Y)\right)^*.\]
	Arguing as in the previous item, for a given $R>0$, the set $\M_{b,\infty}(X,B_Y)_R$ is weak-star compact; or equivalently, it is compact with respect to the topology given by
	\[
	\Phi_\alpha\to \Phi \quad\textrm{ whenever } \quad \Phi_\alpha (f)(y)\to \Phi(f)(y), \textrm{  for all } f\in \H_b(X) \textrm{ and } y\in B_Y.
	\]
\end{enumerate}

\section{Riemann domain over $ \mathcal{H}^{\infty}(B_Y,X^{**}) $}\label{Section-Riemann domain}

For a symmetrically regular Banach space $X$ it is shown in \cite{DiGaMaSe} that $\mathcal{M}_{b,\infty}(X,B_Y)$ can be endowed with a structure of a Riemann domain over $\mathcal L(X^*, Y^*)$. Now, to achieve a fibered description of $\M_\infty(B_X,B_Y)$ we find it more suitable to define a Riemann domain structure over $ \mathcal{H}^{\infty}(B_Y,X^{**}) $ as opposed to $\mathcal L(X^*, Y^*)$. In this way we are choosing  a more complex underlying space but we are dealing with a simpler projection  and the behaviour of the fibers shows more akin to what happens in the scalar-valued case.

As in \cite[Equation (81)]{DiGaMaSe}, for each $g \in \mathcal{H}^{\infty}(B_Y,X^{**}) $ there is an associated composition homomorphism $  C_g \in \mathcal{M}_{b,\infty}(X,B_Y) $ given by
	\[  C_g (f) (y) = \widetilde{f} \circ g(y),\]
	where $\widetilde{f}$ denotes the canonical extension of $f$. It is easily verified that $ C_g$ is well defined and that gives an inclusion
		\begin{align*}
			j:\mathcal{H}^{\infty}(B_Y,X^{**}) &\to \mathcal{M}_{b,\infty}(X,B_Y), \\
			g &\mapsto  C_g.
		\end{align*}
	
Also, as in \cite[Equation (83)]{DiGaMaSe}, there is a projection
		\begin{align*}
			\xi \colon \mathcal{M}_{b,\infty}(X,B_Y) &\to \mathcal{H}^\infty(B_Y,X^{**}), \\
			\Phi & \mapsto \left[ y \mapsto [x^* \mapsto \Phi(x^*)(y)] \right].
		\end{align*}
The fact that this mapping is well-defined can be seen as follows: in order to prove that $\xi(\Phi):B_Y \to X^{**}$ is holomorphic,  it is enough to check that it is weak-star holomorphic \cite[Exercise 8.D]{MujicaLibro}. This is true since, for each $x^*\in X^*$,  we have that
	\begin{align*}
		 \widetilde{x^*}(\xi(\Phi)) = \Phi(x^*) \in \mathcal{H}^\infty(B_Y).
	\end{align*}
 To see that  it is bounded recall that as $\Phi$ belongs to $\mathcal{M}_{b,\infty}(X,B_Y)$, there exists $r>0$ such that $\|\Phi(h)\|\le \|h\|_{rB_X}$, for all $h\in\H_b(X)$. Therefore,
		\begin{align*}
			\sup_{y\in B_Y} \|\xi(\Phi)(y)\| &= \sup_{y\in B_Y} \sup_{\|x^*\| \le 1} |\xi(\Phi)(y)(x^*)| \\
			&=\sup_{\|x^*\| \le 1} \sup_{y\in B_Y} |\Phi(x^*)(y)|\\
			 &\leq \sup_{\|x^*\| \le 1}\|x^*\|_{rB_X} \leq r.
		\end{align*}
Thus, we conclude that $\xi(\Phi)$ belongs to $\mathcal{H}^\infty(B_Y,X^{**})$ (and $\Phi$ is well-defined). Note also that $\xi(j(g))=g$, for all $g \in \mathcal{H}^{\infty}(B_Y,X^{**}) $.

Now, the construction of the Riemann domain structure is analogous to  what was done in the scalar-valued case \cite{AronColeGamelin, AronGalindoGarciaMaestre, Dineen} and in the already mentioned vector-valued case \cite{DiGaMaSe}. Anyway, we present it adapted to our particular setting for future reference throughout the article.
	
For $x^{**} \in X^{**}$,  $\tau_{x^{**}}$ is the translation mapping given by $\tau_{x^{**}} (x) = J_Xx +x^{**}$. This induces, as usual, a mapping $\tau^*_{x^{**}} \colon \mathcal{H}_b(X) \to \mathcal{H}_b(X)$ where $\tau^*_{x^{**}} (f) (x) = \widetilde{f}(J_Xx +x^{**})$. By \cite[Proposition 6.30]{Dineen} for any fixed $f \in \mathcal{H}_b(X)$, we have that the mapping $[x^{**}\mapsto \tau_{x^{**}}^*(f)]$ is a holomorphic function of bounded type from $X^{**}$ into $\mathcal{H}_b(X)$. Then, we can define for each $g \in \mathcal{H}^\infty(B_Y,X^{**})$ and $\Phi \in \mathcal{M}_{b,\infty}(X,B_Y)$ the homomorphism $\Phi^{g}$ in $\mathcal{M}_{b,\infty}(X,B_Y)$ by
		\begin{align*}
			\Phi^g(f)(y) &= \Phi(\tau^*_{g(y)}(f)) (y) \\
			&=\Phi [x \mapsto \widetilde{f}(J_Xx+g(y)) ] (y),
		\end{align*}
	for all $f \in \mathcal{H}_b(X)$ and $y \in B_Y$. In order to see that $\Phi^g$ is well-defined we need  to check that $\Phi^g(f)$ belongs to $\H^\infty (B_Y)$, for every $f \in \mathcal{H}_b(X)$.
	 To derive that $\Phi^g(f)$ is holomorphic, we consider the following function of two variables
		\begin{align*}
			B_Y \times B_Y &\to \mathbb{C} \\
			(y,z) &\mapsto \Phi (\tau^*_{g(z)}(f)) (y).
		\end{align*}
Clearly, the mapping is holomorphic in the first variable. The same is true for the second one, as it is the result of applying $\delta_y \circ \Phi$ to the composition of $[x^{**}\mapsto \tau_{x^{**}}^*(f)]$ with $[y \mapsto g(y)]$. By Hartogs' theorem, this mapping is holomorphic when considering both variables simultaneously and thus it remains so when restricted to the set $\{(y,y):\, y\in B_Y\}$. This gives that the mapping $\Phi^g(f)$ is holomorphic. As before, recall that there exists $r>0$ such that $\|\Phi(h)\|\le \|h\|_{rB_X}$, for all $h\in\H_b(X)$, implying that $\Phi^g(f)$ is bounded:	
		\begin{align}
			\sup_{y\in B_Y} |\Phi^g(f)(y)| &= \sup_{y\in B_Y} |\Phi(\tau^*_{g(y)}(f)) (y)|  \leq \sup_{ y,z\in B_Y  } |\Phi(\tau^*_{g(z)}(f))(y)| \nonumber \\
			 & = \sup_{z\in B_Y} \|\Phi(\tau^*_{g(z)}(f))\|_{B_Y}\leq \sup_{z\in B_Y}\|\tau^*_{g(z)}(f)\|_{rB_X} \label{bounded} \\
			&= \sup_{\substack{z\in B_Y \\ x\in rB_X}} |\widetilde{f}(g(z)+J_Xx)| \leq \|f\|_{(\|g\|+r)B_X} < \infty. \nonumber
		\end{align}
This shows that $\Phi^g$ is well defined. An easy computation shows that $\Phi^g$ is an algebra homomorphism from which we obtain that $\Phi^g \in \mathcal{M}_{b,\infty}(X,B_Y)$.
	
It is also worth noting that the projection $\xi$ satisfies
\begin{equation} \label{xi-phig} \xi(\Phi^g) = \xi(\Phi) + g.
\end{equation}
Indeed, this is clear since for all $y\in B_Y$ and all $x^*\in X^*$, $\tau_{g(y)}^*(x^*)=x^*+g(y)(x^*)$.

We now arrive at the statement of the Riemann domain structure on $\mathcal{M}_{b,\infty}(X,B_Y)$.

		\begin{proposition}
			If $X$ is a symmetrically regular Banach space and $Y$ is any Banach space, $(\mathcal{M}_{b,\infty}(X,B_Y),\xi)$ is a Riemann domain over $\mathcal{H}^\infty(B_Y,X^{**})$ with each connected component homeomorphic to $\mathcal{H}^\infty(B_Y,X^{**})$.
		\end{proposition}
		\begin{proof}
			For $\Phi \in \M_{b,\infty}(X,B_Y)$ and $\varepsilon > 0$, consider the sets
				\[ V_{\Phi,\varepsilon} = \{ \Phi^g \colon g\in \mathcal{H}^\infty(B_Y,X^{**}),\ \|g\| < \varepsilon \}. \]
		These sets form a neighborhood basis for a Hausdorff topology in $\mathcal{M}_{b,\infty}(X,B_Y)$.
		First of all, given $\Psi \in V_{\Phi,\varepsilon}$ we have that $\Psi = \Phi^g$ for a certain $g$ with  $\|g\| < \varepsilon$. Since $X$ is symmetrically regular, by \cite[Lemma 6.28]{Dineen}, we have that $\tau^*_{g(y)} \circ \tau^*_{h(y)}=\tau^*_{(g+h)(y)}$ for all $y\in B_Y$, $h \in \mathcal{H}^\infty (B_Y,X^{**})$. It follows that
			\begin{align*}
				\Psi^h(f)(y) &= (\Phi^g)^h(f)(y) = \Phi^g(\tau^*_{h(y)} (f)) (y)= \Phi (\tau^*_{g(y)} \circ \tau^*_{h(y)} (f)) (y) \\
				& = \Phi(\tau^*_{(g+h)(y)} (f)) (y)= \Phi^{g+h}(f) (y).
			\end{align*}
		Therefore, for $\delta = \varepsilon - \|g\|$ we have that $V_{\Psi,\delta} \subset V_{\Phi, \varepsilon}$.
		That this is in fact  a Hausdorff topology follows as usual: given $\Psi \not= \Phi \in \mathcal{M}_{b,\infty}(B_Y,X^{**})$ there are two possibilities, either $\xi(\Psi) = \xi(\Phi)$ or $\xi(\Psi) \not = \xi(\Phi)$. In the former case, a simple argument using \eqref{xi-phig} implies that $V_{\Phi,\varepsilon} \cap V_{\Psi,\varepsilon'} = \varnothing$ for every $\varepsilon, \varepsilon' > 0$. If, otherwise, $\xi(\Phi) \not = \xi (\Psi)$, the easily obtained conclusion is that $V_{\Phi,\varepsilon} \cap V_{\Psi,\varepsilon} = \varnothing$ for $ \varepsilon = \frac{\|\xi(\Phi) - \xi(\Psi)\|}{2}$.
		Additionally, note that if we consider, for $\Phi \in \mathcal{M}_{b,\infty}(X,B_Y)$, the set
			\[ V_\Phi = \cup_{\varepsilon > 0} V_{\Phi,\varepsilon}=\{\Phi^g \colon g\in \mathcal{H}^\infty(B_Y,X^{**})\}, \]
		we obtain exactly the connected component of $\Phi$, which is homeomorphic to $\H^\infty(B_Y,X^{**})$ as stated.
		\end{proof}
		
		As in \cite[Proposition 10]{DiGaMaSe} each function $f \in \mathcal{H}_b(X)$ can be extended to a function on $\mathcal{M}_{b,\infty}(X,B_Y)$ by means of a sort of Gelfand transform:
			\begin{align*}
				\widehat{f} \colon \mathcal{M}_{b,\infty}(X,B_Y) &\to \mathcal{H}^\infty(B_Y) \\
				\Phi &\mapsto \Phi (f),
			\end{align*}
		and  this function, when restricted to each connected component is a holomorphic function of bounded type. Even though the connected components here are not the same as those in \cite{DiGaMaSe}, the proof developed in that paper works in our context with slight modifications. However, we choose to present here another simpler argument.
		
		\begin{proposition} \label{gelfand-transform}
			Let $X$ be a symmetrically regular Banach space and let $Y$ be any Banach space. Given a function $f \in \mathcal{H}_b(X)$ we have that the extension $\widehat{f}\colon \mathcal{M}_{b,\infty}(X,B_Y) \to \mathcal{H}^\infty(B_Y)$ is a holomorphic function of bounded type, when restricted to each connected component of the spectrum. That is, $\widehat{f} \circ (\xi|_{V_\Phi})^{-1} \in \mathcal{H}_b(\mathcal{H}^\infty(B_Y,X^{**}),\mathcal{H}^\infty(B_Y))$ for every $\Phi\in\M_{b,\infty}(X, B_Y)$.
		\end{proposition}

The proof follows readily by using the following lemma, which is surely known. We include the proof as we could not find a proper reference.

\begin{lemma} \label{holomorphic mapping}
Given Banach spaces $X$ and $Y$ and an open set $U\subset X$, let $F:U\to \mathcal{H}^\infty(B_Y)$ be a locally bounded mapping. Then, $F$ is holomorphic if and only if $\delta_y\circ F$ is holomorphic, for all $y\in B_Y$.

\end{lemma}
		
		\begin{proof}
Since $ \mathcal{H}^\infty(B_Y)$ is the dual of $\mathcal{G}^\infty(B_Y)$,  we have that $F$ is holomorphic if and only if it is so when applied to any element of  $\mathcal{G}^\infty(B_Y)$. To derive the conclusion we have to prove that it is enough to consider just the evaluations $\delta_y\in \mathcal{G}^\infty(B_Y)$. Any $v\in \mathcal{G}^\infty(B_Y)$ can be represented in the quotient $\ell_1(B_Y)/\H^\infty(B_Y)^\perp$ by a sum $\sum_k \lambda_k \delta_{y_k}$, where $(\lambda_k)\in \ell_1$ and $y_k\in B_Y$ for all $k$. So we need to prove that the mapping $[x\mapsto F(x)(v)=\sum_k \lambda_k F(x)(y_k)]$ is holomorphic. Since, by hypothesis, the mappings $[x\mapsto \sum_{k=1}^n \lambda_k F(x)(y_k)]$ are holomorphic for all $n$, to derive the result by means of \cite[Theorem 14.16]{Chae}, we only need to check that these mappings are locally uniformly bounded. Given $x_0\in U$ take $r>0$ and $C>0$ such that $B_X(x_0,r)\subset U$ and $F$ is bounded by $C$ in $B_X(x_0,r)$. Then, for every $x\in B_X(x_0,r)$,
\[
\left| \sum_{k=1}^n \lambda_k F(x)(y_k)\right|\le \sum_{k=1}^n |\lambda_k| |F(x)(y_k)|\le C \|(\lambda_k)\|_{\ell_1}.
\]
\end{proof}

Now we proceed with the proof of the holomorphic property of the Gelfand transform.

\begin{proof} {\it (of Proposition \ref{gelfand-transform})}
We want to prove that the function
				\begin{align*}
					\mathcal{H}^\infty(B_Y,X^{**}) &\to \mathcal{H}^\infty(B_Y) \\
					g &\mapsto \Phi^g (f)
				\end{align*}
			is holomorphic of bounded type. As in equation \eqref{bounded} we obtain that there exists $r>0$ such that
\[
\|\Phi^g (f)\| \leq \|f\|_{(\|g\|+r)B_X}
\] and hence our target function is bounded on bounded sets.
Hence, it is locally bounded. Now, appealing to the previous lemma, it remains to prove that, for all $y\in B_Y$, the mapping $[g\mapsto \Phi^g (f)(y)]$ is holomorphic. This is true since it is the composition of the following two holomorphic mappings:
\[\begin{array}{rclrcl}
\H^\infty(B_Y, X^{**}) &\to & X^{**} &\qquad\qquad X^{**} &\to &\mathbb C\\
g &\mapsto & g(y) &\qquad\qquad x^{**} &\mapsto & \Phi(\tau^*_{x^{**}}(f))(y),
\end{array}
\] and the proof is finished.
\end{proof}

\section{The fibering of $\mathcal{M}_{b,\infty}(X,B_Y)$ over $\mathcal{H}^\infty(B_Y,X^{**})$ } \label{Section-Fiber Mb}
	
	We now focus on the set of elements in $\mathcal{M}_{b,\infty}(X,B_Y)$ that are projected to the same function $g$ of $\mathcal{H}^\infty(B_Y,X^{**})$. This is called the {\em fiber} over $g$ and is defined by
		\[ \mathscr{F}(g) = \{\Phi \in \mathcal{M}_{b,\infty}(X,B_Y) \colon \xi(\Phi) = g \}. \]

	Our aim in this section is to study the size of these sets.

In the scalar-valued spectrum the usual projection is $\pi:\M_b(X)\to X^{**}$, given by $\pi(\varphi)(x^*)=\varphi(x^*)$, for all $x^*\in X^*$. The fiber over each $z\in X^{**}$ is the set of all $\varphi$ such that $\pi(\varphi)=z$. Clearly, the fiber over $z$ contains at least the evaluation homomorphism $\delta_z$.  When finite type polynomials are dense in $\mathcal{H}_b(X)$, the fiber over $z$ is just $\{\delta_z\}$ \cite[Theorem 3.3]{AronColeGamelin}.

Analogously, for every $g \in \mathcal{H}^\infty(B_Y,X^{**})$ we can define the corresponding composition homomorphism $ C_g$. Since $ C_g$ verifies that $\xi( C_g) = g$ we have that the sets $\mathscr{F}(g)$ are non-empty. Moreover, as in the scalar-valued spectrum, the density of finite type polynomials on $X$ implies that the homomorphisms $ C_g$ should be all we find in each fiber. This similarity between scalar and vector-valued spectra is made clear through the following remark.

\begin{remark}\rm \label{Composition_morphism}
For each $\Phi \in \mathcal{M}_{b,\infty}(X,B_Y)$ and each $y\in B_Y$, we denote by $\delta_y\circ \Phi\in\M_b(X)$ the mapping given by $[f\in\H_b(X)\mapsto \Phi(f)(y)]$. Then, it is clear that $\Phi$ is the  composition homomorphism $ C_g$ if and only if for every $y\in B_Y$, $\delta_y\circ \Phi$ is the evaluation homomorphism $\delta_{g(y)}$. Also, $\Phi\in\F(g)$ if and only if for every $y\in B_Y$, $\delta_y\circ \Phi$ is in the fiber (relative to the spectrum $\M_b(X)$) over $g(y)\in X^{**}$.
\end{remark}

Now, we easily obtain the following which was previously observed in \cite[page 10]{DiGaMaSe}.
	
		\begin{proposition}
			Let $X$ and $Y$ be Banach spaces. If finite type polynomials are dense in $\mathcal{H}_b(X)$ then for each $g \in \mathcal{H}^\infty(B_Y,X^{**})$ we have that $\mathscr{F}(g)$ consists solely of the corresponding $ C_g$.
		\end{proposition}
	Whenever finite type polynomials are not dense in $\mathcal{H}_b(X)$ we might find more elements in the fibers over elements in $\mathcal{H}^\infty(B_Y,X^{**})$. For instance, the following theorem shows that if there is a polynomial in $X$ which is not weakly continuous on bounded sets there is a {\em disk} of  homomorphisms in each fiber. The proof is inspired by an analogous result for the scalar-valued spectrum \cite[Theorem 3.1]{AronFalcoGarciaMaestre}.

		\begin{theorem} \label{inyectar disco}
			If $X$ is a Banach space such that there exists a polynomial on $X$ which is not weakly continuous on bounded sets, then for each $g \in \mathcal{H}^\infty(B_Y,X^{**})$ we can inject the complex disk $\mathbb{D}$ analytically into the fiber $\mathscr{F}(g)$.
		\end{theorem}
		\begin{proof} If there exists a polynomial on $X$ which is not weakly continuous on bounded sets, then (for a certain $m$) there is an $m$-homogeneous polynomial $P$ such that its canonical extension $\widetilde{P}$ is not weak-star continuous at any $x^{**} \in X^{**}$ (see \cite[Corollary 2]{BoydRyan} or \cite[Proposition 1]{AronDimant}).
			Given $g \in \mathcal{H}^\infty(B_Y,X^{**})$, denoting $x^{**}_0 = g(0)$
			 we can find an $\varepsilon > 0$ and a bounded net $(x^{**}_{\alpha})$, weak-star convergent to $x^{**}_0$, such that $|\widetilde{P}(x^{**}_\alpha) - \widetilde{P}(x^{**}_0)| > \varepsilon$  for every $\alpha$.
			We now fix an ultrafilter $\mathscr{U}$ containing the sets $\{ \alpha \colon \alpha \geq \alpha_0 \}$ and define for $t \in \D$ the mapping $\Phi_t: \mathcal{H}_b(X) \to \mathcal{H}^\infty(B_Y)$ by
			\[
			\Phi_t(f)(y) = \lim_{\mathscr{U}} \widetilde{f}(g(y) + t(x^{**}_\alpha - x^{**}_0 )).
			\]
Note that the limit along the ultrafilter exists because for each $t$ and each $f$, if $M$ is a bound for the sequence $(x^{**}_{\alpha})$ we have
\[
\left\|\widetilde{f}(g(y) + t(x^{**}_\alpha - x^{**}_0 ))\right\|\le\|f\|_{(\|g\|+M+\|x_0^{**}\|)B_X},
\] and so the set $(\mathcal{M}_{b,\infty}(X,B_Y))_{\|g\|+M+\|x_0^{**}\|}$ is weak-star compact (see item \ref{comment3} of the comment about duality and compactness in the Introduction).
 The previous inequality also shows that, for all $\alpha$, the mappings $[y\mapsto \widetilde{f}(g(y) + t(x^{**}_\alpha - x^{**}_0 ))]$ are in a ball of $\H^\infty(B_Y)=\left(\mathcal{G}^\infty(B_Y)\right)^*$ and by weak-star compactness we obtain that $\Phi_t(f)\in\H^\infty(B_Y)$.
			Also, it is easy to see that, for each $t \in \mathbb{C}$, $\Phi_t$ is a homomorphism in  $\mathcal{M}_{b,\infty}(X,B_Y)$ and  $\xi(\Phi_t) = g$. 	 To assert that the mapping $[t\mapsto \Phi_t]$ is analytic, we need to check that for every $f \in \mathcal{H}^\infty(B_X)$ the following mapping is analytic:
						\begin{align*}
			\Phi(f): \mathbb{D} &\to \mathcal{H}^\infty(B_Y) \\
			t&\mapsto \Phi_t(f) = \lim_{\mathscr{U}} \widetilde{f}(g(y) + t(x^{**}_\alpha - x^{**}_0 )).
			\end{align*}
						Fix $f \in \mathcal{H}^\infty(B_X)$ and define $f_\alpha: \mathbb{D} \to\mathcal{H}^\infty(B_Y)$ by
			\[ f_\alpha (t) (y) = \widetilde{f}(g(y) + t(x^{**}_\alpha - x^{**}_0 )).  \]
			The set $\{f_\alpha\}_\alpha$ is contained in $\|f\|_{KB_X}\overline{B}_{\mathcal{H}^\infty(\mathbb{D},\mathcal{H}^\infty(B_Y))}$, where $K=\|g\|+(|t_0|+s)(M+\|x_0^{**}\|)$. Since, by \cite[Theorem 2.1]{Mujica},
			\[ \mathcal{H}^\infty(\mathbb{D},\mathcal{H}^\infty(B_Y)) = \mathcal{L}(\mathcal G^\infty(\mathbb{D}),\mathcal{H}^\infty(B_Y)) = (\mathcal G^\infty(\mathbb{D})\widehat{\otimes}_\pi \mathcal G^\infty(B_Y))^*, \]
			the set $\|f\|_{KB_X}\overline{B}_{\mathcal{H}^\infty(\mathbb{D},\H^\infty(B_Y))}$ is a weak-star compact set, which tells us that the limit of the $f_\alpha$'s can be taken analytically on $t$.
			This proves the analyticity of the mapping
			\begin{align*}
			t&\mapsto \Phi_t(f) = \lim_{\mathscr{U}} \widetilde{f}(g(y) + t(x^{**}_\alpha - x^{**}_0 )).
			\end{align*}
					Moreover, since $P$ is an $m$-homogeneous polynomial, we can write
			\[ \widetilde{P}(g(y) + t(x^{**}_\alpha - x^{**}_0 )) = \widetilde{P}(g(y)) + \sum_{j=1}^m t^j \binom{m}{j} \v{\widetilde{P}}(g(y)^{m-j},(x^{**}_\alpha - x^{**}_0)^j). \]
			In particular, taking $y=0$ we obtain
			\[ \Phi_t(P)(0) = \sum_{j=0}^m a_j t^j. \]
			Now, since
			\[ |\Phi_1(P)(0) - \Phi_0(P)(0)| = \lim_{\mathscr{U}}|\widetilde{P}(x^{**}_\alpha) - \widetilde{P}(x^{**}_0)| \geq \varepsilon, \]
			we have a non-constant polynomial of degree $\leq m$, so we can find $t_0 \in \D$ and $s > 0$ such that $\Phi_t(P)(0)$ is injective in $\mathbb{D}(t_0,s)$.
		Finally,  through the composition with the mapping $\gamma :\D \to \D(t_0,s)$ given by $[t\mapsto t_0+st]$ we obtain that $\Phi\circ\gamma:\D\to \F(g)$ is the desired analytic injection.
		\end{proof}

\begin{remark} \rm \textbf{Fibers over constant functions.} \label{Remark: constant} The scalar-valued spectrum $\M_b(X)$ is naturally seen inside $\M_{b,\infty}(X,B_Y)$ through the inclusion mapping $[\varphi \mapsto \varphi\cdot1_Y]$ where each element of $\M_b(X)$ lies in a fiber over a constant function. Then, for a constant function $g\in \mathcal{H}^\infty(B_Y,X^{**})$ it is natural to wonder whether there are homomorphisms in $\F(g)$ not belonging to $\M_b(X)$. It is worth noting that the previous theorem does not provide examples of that kind.  Indeed, if we take $g(y)=x_0^{**}$ for all $y$, we have that
		\begin{align*}
			\Phi_t (f) (y) = \lim_{\mathscr{U}} \widetilde{f} (x_0^{**} + t(x_\alpha^{**} - x_0^{**})),
		\end{align*}
	which is a constant function of $y$ and hence identified with an element of $\M_b(X)$. However, building on the previous result we obtain in the next theorem an analytic injection of the ball $B_{\H^\infty(B_Y)}$ in each fiber over a constant function, providing examples of non scalar-valued homomorphisms in those fibers.
\end{remark}

\begin{theorem} \label{Prop:constant_function}
			If $X$ is a Banach space such that there exists a polynomial on $X$ which is not weakly continuous on bounded sets, then for each constant function $g \in \mathcal{H}^\infty(B_Y,X^{**})$ we can inject the ball $B_{\H^\infty(B_Y)}$ analytically in the fiber $\mathscr{F}(g)$. Moreover, through this inclusion each non-constant function in $B_{\H^\infty(B_Y)}$ is mapped into a non scalar-valued homomorphism of $\mathscr{F}(g)$.
		\end{theorem}
		\begin{proof} Given $g(y)=x_0^{**}$ for all $y$, from the proof of Theorem \ref{inyectar disco} we have an analytic injection $\Phi\circ\gamma:\D\to \F(g)$. Now consider $\Psi:B_{\H^\infty(B_Y)}\to\F(g)$ given by
\[
\Psi(h)(f)(y)= \Phi\circ\gamma (h(y))(f), \quad\textrm{ for all } h\in B_{\H^\infty(B_Y)},\ f\in\H_b(X),\ y\in B_Y.
\] Note that this definition makes sense because, by the previous remark, for each $h$ and $y$,  the homomorphism $\Phi\circ\gamma (h(y))$ is scalar-valued.

Let us check that $\Psi$ is well defined, analytic and injective. First, note that, for all $f\in\H_b(X)$ and $h\in B_{\H^\infty(B_Y)}$, the mapping $\Psi(h)(f)$  is analytic since it is the composition of two holomorphic mappings: $[y\mapsto h(y)]$ and $[t\mapsto \Phi\circ\gamma (t)(f)]$. As before, it is bounded:
\[
\sup_{y\in B_Y}|\Psi(h)(f)(y)|\le \|f\|_{(\|g\|+M+\|x_0^{**}\|)B_X}.
\]
Now, it is readily seen that $\Psi(h)$ belongs to $\M_{b,\infty}(X,B_Y)$ and that in fact it is in the fiber over $g$, so $\Psi$ is well defined.

Secondly, for each $f\in\H_b(X)$, the mapping from $B_{\H^\infty(B_Y)}$ to $\H^\infty(B_Y)$ given by $[h\mapsto \Psi(h)(f)]$ is analytic. Indeed, by Lemma \ref{holomorphic mapping}, it is enough to see that the mapping $[h\mapsto\Psi(h)(f)(y)]$ is analytic, which again can be done by writing it as the composition of a linear and a holomorphic mapping: $[h\mapsto h(y)]$ and $[t\mapsto \Phi\circ\gamma (t)(f)]$.

Thirdly,  $\Psi$ is injective because $\Phi\circ\gamma$ has the same property.

Finally, note that $\Psi$ maps each non-constant function in $B_{\H^\infty(B_Y)}$ into a non scalar-valued homomorphism. If $h\in B_{\H^\infty(B_Y)}$ is non-constant then there exist $y_1$ and $y_2$ in $B_Y$ such that $h(y_1)\not = h(y_2)$ and thus $\Phi\circ\gamma (h(y_1))\not = \Phi\circ\gamma (h(y_2))$. So $\Psi(h)$ cannot be of the form $\varphi\cdot1_Y$ for a scalar valued $\varphi$.
\end{proof}

	\section{The radius function} \label{Section-Radius function}

Aron, Cole and Gamelin \cite{AronColeGamelin} introduced a radius function on $\M_b(X)$ and proved several properties. Then, they extended this definition to homomorphisms in $\M_\infty(B_X)$  establishing a relationship between both spectra. We now follow the same plan in the vector-valued case.
	
	Given a homomorphism $\Phi \in \mathcal{M}_{b,\infty}(X,B_Y)$ we define its radius as
	\[ R(\Phi) = \inf \{ r > 0 \colon \|\Phi(f)\|_{B_Y} \leq \|f\|_{rB_X},  f \in \mathcal{H}_b(X) \}.\]
	It is worth noting that since the homomorphisms in $\mathcal{M}_{b,\infty}(X,B_Y)$ are continuous we have $ 0 \leq R(\Phi) < \infty$. Furthermore, the following result regarding the continuity of $\Phi$ in $R(\Phi)B_X$ holds. Note that  this is a vector-valued version of \cite[Lemma 2.1]{AronColeGamelin}. We omit the proof, as it is identical.
	
		\begin{lemma}
			For every $\Phi \in \mathcal{M}_{b,\infty}(X,B_Y)$ and $f \in \mathcal{H}_b(X)$  we have
			\[ \|\Phi(f)\|_{B_Y} \leq \|f\|_{R(\Phi)B_X}.
			 \]
		\end{lemma}
		
		
For $\Phi \in \mathcal{M}_{b,\infty}(X,B_Y)$ we denote by $\Phi_m$ its restriction to $\mathcal P(^mX)$, that is $\Phi_m$ is a linear operator from $\mathcal P(^mX)$ into $\H^\infty(B_Y)$. As in the scalar-valued case we have

		\begin{proposition}
			\label{RadioLimsup}
		The radius function $R$ on $\mathcal{M}_{b,\infty}(X,B_Y)$ is given by
			\begin{equation*}
			R(\Phi) = \limsup_{m \to \infty} \|\Phi_m\|^{1/m}=\sup_{m \ge 1} \|\Phi_m\|^{1/m}.
			\end{equation*}
		\end{proposition}
		
		\begin{proof}

The first equality follows the lines of the proof of \cite[Theorem 2.3]{AronColeGamelin}. For the second one we need a slight change in the argument. It is observed in \cite[page 55]{AronColeGamelin} that
\[
\|\varphi_m\|^2\le\|\varphi_{2m}\|\qquad\textrm{ for all } \varphi\in\M_b(X) \textrm{ and  }m\in\mathbb N.
\]
Thus the same is true for vector-valued homomorphisms. Hence,
\[
\|\Phi_m\|^{1/m}\le\|\Phi_{2m}\|^{1/2m}\qquad\textrm{ for all } \Phi\in\M_{b,\infty}(X,B_Y) \textrm{ and  }m\in\mathbb N,
\]
which implies that the limit superior should coincide with the supremum.
		\end{proof}

Note that the limit superior above is not necessarily a limit. In \cite{Deghoul} Deghoul exhibits an example of an homomorphism $\varphi$ in $\M_b(\ell_2)$ with $R(\varphi)\not=0$ and $\|\varphi_m\|=0$ for every odd $m$.

	\begin{remark}\rm
As we have already observed,  the role played by the evaluation homomorphisms $\delta_z$ in the scalar-valued spectrum is performed here by the composition homomorphisms $ C_g$. It is easy to see that  vector-valued versions of \cite[Lemma 3.1 and Lemma 3.2]{AronColeGamelin} are valid:
				\[\|\xi(\Phi)\| \leq R(\Phi), \ \Phi \in \mathcal{M}_{b,\infty}(X,B_Y), \]
			and also
				\[R( C_g) = \|g\|, \ g \in \mathcal{H}^{\infty}(B_Y,X^{**}). \]
		\end{remark}
			

Let us now translate the radius function to the spectrum $\mathcal{M}_\infty(B_X,B_Y)$. For that, first we consider the	 natural projection
			\[ \varrho \colon \mathcal{M}_\infty(B_X,B_Y) \to \mathcal{M}_{b,\infty}(X,B_Y), \]
		defined so that $\varrho(\Psi)$ is the restriction of $\Psi \in \mathcal{M}_\infty(B_X,B_Y)$ to $\mathcal{H}_b(X).$

		We then extend the radius function $R$ to $\Psi \in \mathcal{M}_\infty(B_X,B_Y)$ by declaring $R(\Psi)$ to be the smallest value of $r$, $0 \leq r \leq 1$ such that $\Psi$ is continuous with respect to the norm of uniform convergence on the ball $rB_X$. Applying the previous results of this section and the fact that $\mathcal{M_\infty}(B_X,B_Y)$ is weak-star compact (as noted in item \ref{compact} of the observation regarding duality and compactness in the Introduction), the proof of \cite[Theorem 10.1]{AronColeGamelin} can be easily adapted to our setting, arriving at the following result.
		
		\begin{theorem}
			The image $\varrho(\mathcal{M}_\infty(B_X,B_Y))$ of the projection $\varrho$ consists of precisely the set of $\Phi \in \mathcal{M}_{b,\infty}(X,B_Y)$ such that $R(\Phi) \leq 1$. Moreover, the projection $\varrho$ establishes a one-to-one correspondence between the set of $\Psi \in \mathcal{M}_{\infty}(B_X,B_Y)$ satisfying $R(\Psi) < 1$ and the set of $\Phi \in \mathcal{M}_{b,\infty}(X,B_Y)$ satisfying $R(\Phi) < 1$.
		\end{theorem}

\section{The fibering of $\mathcal{M}_{\infty}(B_X,B_Y)$ over $\overline B_{\mathcal{H}^\infty(B_Y,X^{**})}$ }\label{Section-Minf}

As in the  case of $\mathcal{M}_{b,\infty}(X,B_Y)$,  we can define a natural projection from the vector-valued spectrum $\mathcal{M}_{\infty}(B_X,B_Y)$ into $\mathcal{H}^\infty(B_Y,X^{**})$, by composing $\xi \colon \mathcal{M}_{b,\infty}(X,B_Y) \to \mathcal{H}^\infty(B_Y,X^{**})$ with $\varrho \colon \mathcal{M}_\infty(B_X,B_Y) \to \mathcal{M}_{b,\infty}(X,B_Y)$. In order to simplify the notation we choose to denote this projection again by $\xi$ (instead of $\xi\circ \varrho$). In this setting, $\xi$ is defined by:
\begin{align*}
			\xi \colon \M_\infty(B_X,B_Y) &\to \mathcal{H}^\infty(B_Y,X^{**}), \\
			\Phi & \mapsto \left[ y \mapsto (x^* \mapsto \Phi(x^*)(y)) \right].
		\end{align*}
The image of $\xi$ is clearly contained in the closed unit ball of $\mathcal{H}^\infty(B_Y,X^{**})$. Also, for each $g\in \mathcal{H}^\infty(B_Y,X^{**})$ such that $g(B_Y)\subset B_{X^{**}}$ we can consider the composition homomorphism $ C_g\in \M_\infty(B_X,B_Y)$ given by $ C_g(f)=\widetilde f\circ g$, for all $f\in\H^\infty(B_X)$. Since $\xi( C_g)=g$ the following inclusions hold:
\[
B_{\mathcal{H}^\infty(B_Y,X^{**})}\subset \{g\in \mathcal{H}^\infty(B_Y,X^{**}):\ g(B_Y)\subset B_{X^{**}}\}\subset Im(\xi).
\]
Note that, as we have already mentioned, by \cite[Theorem 2.1]{Mujica} the space $\mathcal{H}^\infty(B_Y,X^{**})$ is isometric to $\mathcal L(\mathcal G^\infty(B_Y), X^{**})$ and so it is the dual of $\G^\infty(B_Y)\widehat\otimes_\pi X^*$.
Now, as $\mathcal{M}_{\infty}(B_X,B_Y)$ is weak-star compact and $\xi$ is weak-star to weak-star continuous, the image of $\xi$ should be weak-star compact in $\mathcal{H}^\infty(B_Y,X^{**})$. Hence
\[
Im(\xi)=\overline {B}_{\mathcal{H}^\infty(B_Y,X^{**})}.
\]

Now, we turn our attention to the fibers defined by this projection. For $g\in \overline B_{\mathcal{H}^\infty(B_Y,X^{**})}$, the fiber over $g$ is the set
\[
 \mathscr{F}(g) = \{\Phi \in \mathcal{M}_{\infty}(B_X,B_Y) \colon \xi(\Phi) = g \}.
\]

For the scalar-valued spectrum $\mathcal{M}_{\infty}(B_X)$, to study the fibers over $\overline B_{X^{**}}$, the distinction between points $z$ in the interior of the ball (for which the evaluation $\delta_z$ is in the fiber) and points $z$ in the boundary (where $\delta_z$ cannot be defined) is relevant. In the vector-valued case recall that for a holomorphic function $g\in \overline B_{\mathcal{H}^\infty(B_Y,X^{**})}$ if $\|g(y_0)\|=1$ for a certain $y_0\in B_Y$ then $g(y)$ belongs to $S_{X^{**}}$ (the unit sphere of $X^{**}$) for all $y\in B_Y$. Thus, to distinguish the fibers over $g\in \overline B_{\mathcal{H}^\infty(B_Y,X^{**})}$ in terms of whether $ C_g$ is or is not defined, we get the following two possibilities for $g$:
\begin{enumerate}
\item[(i)] $g(B_Y)\subset B_{X^{**}}$ (where $ C_g\in  \mathscr{F}(g)$).
\item[(ii)] $g(B_Y)\subset S_{X^{**}}$ (where $ C_g$ cannot be defined).
\end{enumerate}

Note that whenever $X^{**}$ is strictly convex, the only functions $g\in \overline B_{\mathcal{H}^\infty(B_Y,X^{**})}$ with $g(B_Y)\subset S_{X^{**}}$ are the constant functions. Then, in this case  the condition (ii) changes to:
\begin{enumerate}
\item[(ii')] There exists $x_0^{**}\in S_{X^{**}}$ such that $g(y)=x_0^{**}$, for all $y\in B_Y$ (where $ C_g$ cannot be defined).
\end{enumerate}

Also, as we commented for the spectrum $\M_{b,\infty}(X,B_Y)$ in Remark \ref{Remark: constant}, there are other {\em special } fibers in $\mathcal{M}_{\infty}(B_X,B_Y)$ which are  those over constant functions $g\in \overline B_{\mathcal{H}^\infty(B_Y,X^{**})}$. Each of these fibers includes the scalar-valued fiber of $\M_\infty(B_X)$ over the same constant.

Recall that it is said that an element $x_0^{**}\in S_{X^{**}}$ is {\em norm attaining} if there exists $x_0^*\in S_{X^*}$ such that $x_0^{**}(x_o^*)=1$.

Now we show that, if $x_0^{**}\in S_{X^{**}}$ is norm attaining,  the fiber over the constant function $g(y)=x_0^{**}$   contains a lot of elements that do not arise from the scalar-valued spectrum. The proof is build on the following proposition from \cite{AronFalcoGarciaMaestre}.

\begin{proposition} \cite[Proposition 2.1]{AronFalcoGarciaMaestre}
	\label{inj-falco}
			Given a Banach space $X$ and a norm attaining element $x_0^{**}\in S_{X^{**}}$, there is an analytic injection
			\[ F: \mathbb{D}  \hookrightarrow \mathcal{M}_{x_0^{**}}(\H^\infty(B_X)).  \]
\end{proposition}			
						
Note that this clearly holds for every $x_0\in S_X$. Now, to transfer this construction to the vector-valued spectrum recall that in Theorem \ref{Prop:constant_function} we have proved a similar result regarding the fibers over constant functions in  the spectrum $\M_{b,\infty}(X,B_Y)$   with the additional hypothesis of the existence of a polynomial that is not weakly continuous on bounded sets. The proof of the following result follows the lines of the proof of Theorem \ref{Prop:constant_function}, and so we omit it.
		\begin{proposition}
			\label{FibrasS_X}
			Given  Banach spaces $X$ and $Y$ and  a norm attaining element $x_0^{**}\in S_{X^{**}}$, let $g(y)=x_0^{**}$, for all $y\in B_Y$. Then there is an analytic injection
							\begin{align*}
					\Psi: B_{\H^\infty(B_Y)} &\hookrightarrow \mathscr{F}(g) \subset \M_\infty(B_X,B_Y)\\
					\Psi(h)(f)(y) &= F(h(y)) (f),
				\end{align*}
where $F$ is the mapping of the previous proposition.
				Moreover, each non-constant function in $B_{\H^\infty(B_Y)}$ is mapped into a non scalar-valued homomorphism of $\mathscr{F}(g)$.
		\end{proposition}
		
%
%

For a finite dimensional $X$ we quote a conjecture from \cite[page 88]{AronColeGamelin}: {\em ``one expects that the fiber over $x_0$ consists of only the evaluation homomorphism $\delta_{x_0}$ for $x_0\in B_{X}$''}. We do not know whether this is true  but this is certainly the case for $B_X=\D$ and for each finite dimensional ball $B_X$ such that the Gleason problem is solved for $\H^\infty(B_X)$ (see \cite[6.6]{Rudin-FunctionTheoryCn} or \cite{carlsson, lemmers} and references therein), for instance, $X=\ell_p^n$, with $1<p<\infty$. In the context of a  strictly convex finite dimensional Banach space $X$ where the Gleason problem is solved for $\H^\infty(B_X)$ we have an almost complete depiction of the fibers of $\M_\infty(B_X,B_Y)$ which resembles the description of the fibers of $\M_\infty(B_X)$. The result, stated in the next theorem, is obtained just summing up the above comments and Proposition \ref{FibrasS_X}. We point out that item $(i)$ was previously proved for the ball of $\ell_2^n$ in \cite[Theorem 6.6.5]{Rudin-FunctionTheoryCn} and for the disk $\D$ in \cite[Proposition 15]{GalindoLindstrom}.

\begin{theorem} \label{propDisk} If $X$ is a strictly convex finite dimensional Banach space such that the Gleason problem is solved for $\H^\infty(B_X)$ then for any given $g \in \overline{B}_{\H^\infty(B_Y, X)}$ there are two alternatives for the fiber $\mathscr{F}(g)$:
			\begin{enumerate}
\item[(i)] If $g(B_Y) \subset B_X$, then $\mathscr{F}(g) = \{ C_g\}.$

				\item[(ii)] If $g \equiv x_0$ with $x_0\in S_{X}$, then $B_{\H^\infty(B_Y)}$ can be analytically injected in $\mathscr{F}(g).$
				
			\end{enumerate}
		\end{theorem}
		
		\begin{proof} First, recall that $X$ being strictly convex implies the only two possibilities for a given $g\in \overline{B}_{\H^\infty(B_Y,X)}$ are those in the previous items.
Now, if $g$ satisfies $(i)$, for every $\Phi \in \mathscr{F}(g)$ and each $y\in B_Y$,  we have that $\delta_y \circ \Phi \in \M_\infty(B_X)$ is in the fiber over $g(y)\in B_X$. Since this fiber is a singleton $\{\delta_{g(y)}\}$, we easily infer that $\mathscr{F}(g) = \{ C_g\}$.
			If, whereas, $g \equiv x_0$ with $x_0\in S_{X}$, the result follows from Proposition \ref{FibrasS_X}.
					\end{proof}

A  lot of research is available in the literature about the basic simplest case  $X=Y=\mathbb C$ (that is, homomorphisms from $\H^\infty(\D)$ into $\H^\infty(\D)$). Anyway, through our construction we can give a slightly different description of the spectrum $\M_\infty(\D, \D)$ resembling the classical scalar-valued situation. Indeed, this vector-valued spectrum is projected onto $\overline B_{\H^\infty(\D)}$, being one-to-one over the set $\{g\in \overline B_{\H^\infty(\D)}:\ g(\D)\subset \D\}$. The remaining fibers (i. e. those over constant functions $g$ of modulus 1) are large, they have plenty of non scalar-valued homomorphisms and each one contains an analytic copy of $B_{\H^\infty(\D)}$.

For any infinite dimensional Banach space $X$ we know from \cite[Theorem 11.1]{AronColeGamelin} that each fiber of the spectrum $\M_\infty(B_X)$ contains a homeomorphic copy of $\beta(\mathbb N)$. This canonically translates to fibers over  constant functions $g\in \overline{B}_{\H^\infty(B_Y,X^{**})}$ in the spectrum $\M_\infty(B_X,B_Y)$. We can extend this result to fibers over (non-constant) functions $g$ of constant norm 1. Recall that $\beta(\mathbb{N})\setminus \mathbb N$ contains a homeomorphic copy of $\beta(\mathbb N)$  so it is enough to obtain a homeomorphic copy of $\beta(\mathbb{N})\setminus \mathbb N$ inside the  fiber.

\begin{proposition}
   			If $X$ is an infinite dimensional Banach space and $g \in \overline{B}_{\H^\infty(B_Y,X^{**})}$ is a function of constant norm $1$, then the fiber in $\M_\infty(B_X,B_Y)$ over $g$  contains a homeomorphic copy of $\beta(\mathbb{N})$.
   		\end{proposition}
   		
   		\begin{proof}
   			Let $g \in \overline{B}_{\H^\infty(B_Y,X^{**})}$ be a function of constant norm 1 and fix $y_0 \in B_Y$. Since $\|g(y_0)\| = 1$, by \cite[Theorem 10.5]{AronColeGamelin}, for each $(r_n)_n\in \D$ with $r_n \to 1$ the sequence $(r_ng(y_0))$ has an interpolating subsequence for $\H^\infty(B_X)$ (which we still call $(r_ng(y_0))$).
   			We now write $g_n = r_ng$ and consider the mapping
   				\begin{align*}
   					I:\mathbb{N} &\to \overline{\{C_{g_n}\}}^{w^*} \subseteq \mathcal{M}_\infty(B_X,B_Y) \\
   					m &\mapsto C_{g_m}.
   				\end{align*}
   			By the universal property of $\beta\mathbb{N}$, there is a continuous extension $\beta I: \beta \mathbb{N} \to \overline{\{C_{g_n}\}}^{w^*}$ such that $\beta I|_{\mathbb{N}} = I$. Since $(r_ng(y_0))$ is an interpolating sequence, the composition $\delta_{y_0} \circ \beta I$ is injective and so must be $\beta I$.
   			
			Finally, a straightforward computation shows that for every $\eta\in \beta(\mathbb{N})\setminus \mathbb N$, the image $\beta I (\eta)$ lies in the fiber over $g$.
   		\end{proof}

From the above, we know that, when $X$ is infinite dimensional, fibers of $\M_\infty(B_X,B_Y)$ over constant functions or over functions of constant norm 1 are large. It is thus natural to ask whether the same is true for the remaining fibers, that is,  fibers over  non-constant functions $g\in \overline{B}_{\H^\infty(B_Y,X^{**})}$ with $g(B_Y)\subset B_{X^{**}}$. We have no general answer to this question. We can only say something under the hypothesis of existence of a polynomial that is not weakly continuous on bounded sets, and even in this case we can just reach the fibers over functions $g \in B_{\H^\infty(B_Y,X^{**})}$, as we present below. We do not know whether this result can be extended to fibers over functions $g$ of norm 1 such that $g(B_Y)\subset B_{X^{**}}$ or not.

\begin{theorem} \label{Inyectar disco Minf}
			If $X$ is a Banach space such that there exists a polynomial on $X$ which is not weakly continuous on bounded sets, then for each $g \in B_{\H^\infty(B_Y,X^{**})}$ we can inject the complex disk $\mathbb{D}$ analytically in the fiber $\mathscr{F}(g)$.
	 	\end{theorem}

The proof is quite similar to that of Theorem \ref{inyectar disco}, so we omit it.  A slight change arises while choosing the net $(x_\alpha^{**})$ which should be taken in the ball $B(x_0^{**},1-\|g\|)$.

Also, mimicking the arguments of Theorem \ref{Prop:constant_function} we have an analogous result for the fibers over constant functions of norm smaller than 1.

\begin{theorem} \label{Inyectar bola Minf}
			If $X$ is a Banach space such that there exists a polynomial on $X$ which is not weakly continuous on bounded sets, then for each constant function $g \in B_{\mathcal{H}^\infty(B_Y,X^{**})}$ we can inject the ball $B_{\H^\infty(B_Y)}$ analytically in the fiber $\mathscr{F}(g)$. Moreover, through this inclusion each non-constant function in $B_{\H^\infty(B_Y)}$ is mapped into a non scalar-valued homomorphism of $\mathscr{F}(g)$.
		\end{theorem}

A typical example of a space where all the polynomials are weakly continuous on bounded sets is the sequence space $c_0$. The study of the algebra $\H^\infty(B_{c_0})$ is interesting also because $B_{c_0}$ is the natural infinite dimensional extension of the polydisk $\D^n$. Even though the previous results do not apply to $X=c_0$, it is anyway possible in this case to insert analytic copies of balls into the fibers of the vector-valued spectrum. This result and a thorough study about $\M_\infty(B_{c_0}, B_{c_0})$ will appear in a forthcoming article \cite{DimantSinger2}.

\section{Gleason parts for $\mathcal{M}_{\infty}(B_X,B_Y)$ }\label{Section-Gleason parts}

The study of Gleason parts in the spectrum of uniform algebras was motivated by the search for analytic structure. This is justified by the fact that the image of an analytic mapping from an open convex set into the spectrum should be contained in a single Gleason part. A thorough description of Gleason parts for the spectrum $\M_\infty(\D)$ was made by K. Hoffman in \cite{Hoffman} and later on deeply studied by several authors (see, for instance, \cite{Gorkin, mortini, Suarez}). For an infinite dimensional Banach space $X$ the study of Gleason parts for $\M_\infty(B_X)$ (with special emphasis on the case $X=c_0$) was initiated in \cite{AronDimantLassalleMaestre}. Let us recall this notion.

For a uniform algebra $\A$ with spectrum $\M(\A)$, the {\em pseudo-hyperbolic distance} in the spectrum is given by $\rho(\varphi, \psi) =
\sup\{|\varphi(f)|:  \  f \in \A, \|f\| \leq 1, \psi(f) = 0 \}$. Note that  $\rho(\varphi, \psi)\le 1$ and that this notion is related to the usual metric in the spectrum by the following known equality \cite[Theorem~2.8]{Bear}:
\begin{equation}\label{Gleason-metric}
\|\varphi - \psi\| = \frac{2 - 2\sqrt{1 - \rho(\varphi, \psi)^2}}{\rho(\varphi,\psi)}.
\end{equation}

This means that $\|\varphi - \psi\| = 2 $ if and only if $\rho(\varphi, \psi) = 1$, motivating
for each $\varphi \in \mathcal M(\mathcal A)$, the definition of the {\em Gleason part} of $\varphi$  as
the set
\[\mathcal {GP}(\varphi) = \{ \psi : \ \rho(\varphi,\psi) < 1\}=\{\psi : \ \|\varphi - \psi\| < 2\}.\]
An interesting aspect here is that these sets form a partition of $\M(\A)$ into equivalence classes.

We can extend the notion of Gleason part to the vector-valued spectrum. Indeed, for $\Phi\in\M_\infty(B_X, B_Y)$ we set
\[\mathcal {GP}(\Phi) = \{\Psi : \ \|\Phi - \Psi\| < 2\}=\{ \Psi : \ \sigma(\Phi, \Psi) =\sup_{y\in B_Y}\rho(\delta_y\circ\Phi,\delta_y\circ\Psi) < 1\}.\]
The equality between the above sets is clear from its analogous statement  in the scalar-valued spectrum (recall that $\delta_y\circ\Phi$ belongs to $\M_\infty(B_X)$ for each $\Phi\in\M_\infty(B_X, B_Y)$ and $y\in B_Y$). It is also readily seen, appealing again to the scalar-valued result, that Gleason parts lead to a partition of $\M_\infty(B_X, B_Y)$ into equivalence classes. This notion, without the  specific name of {\em Gleason parts}, was previously considered in several articles (see, for instance, \cite{AronGalindoLindstrom, chu-hugli-mackey,GalindoGamelinLindstrom,GorkinMortiniSuarez, HozokawaIzuchiZheng,MacCluerOhnoZhao}).

In \cite{AronDimantLassalleMaestre} the relationship between fibers and Gleason parts for $\M_\infty(B_X)$ was addressed.
Some of the results of that article have a vector-valued counterpart that we now present. We begin by proving a version of \cite[Proposition 1.1]{AronDimantLassalleMaestre} that shows which fibers might share Gleason parts.

Here the notation $C_{\mathbf 0}$ refers to the composition homomorphism by the constant function $g\equiv 0$.

\begin{proposition}{\label{basic 1}}
Let $X$ and $Y$ be Banach spaces.
\begin{enumerate}[\upshape (a)]
\item For every $g \in B_{\mathcal{H}^\infty(B_Y,X^{**})}$, the composition homomorphism $ C_g$ is contained in the Gleason part $\mathcal{GP}(C_{\mathbf 0})$. In fact, $\sigma( C_g,C_{\mathbf 0}) = \|g\|$.
\item Let $g \in S_{\mathcal{H}^\infty(B_Y,X^{**})}$ and $h\in B_{\mathcal{H}^\infty(B_Y,X^{**})}$. For any  $\Phi \in \mathscr{F}(g)$ and $\Psi \in \mathscr{F}(h)$ we have that $\Phi$ and $\Psi$ lie in different Gleason parts.
 \item Let $g,\, h\in \mathcal{H}^\infty(B_Y,X^{**})$ with $g(B_Y)\subset B_{X^{**}}$ and $h(B_Y)\subset S_{X^{**}}$. For any   $\Phi \in \mathscr{F}(g)$ and $\Psi \in \mathscr{F}(h)$ we have that
 $\Phi$ and $\Psi$ lie in different Gleason parts.
\end{enumerate}
\end{proposition}

\begin{proof}
(a) By the Schwarz lemma and the fact that $S_{X^*}$ is a norming set for $X^{**}$ we obtain
\begin{eqnarray*}
\rho(\delta_y\circ C_g,\delta_y\circ C_{\mathbf 0}) &= &\sup\{|\delta_y\circ C_g(f)|: f\in\H^\infty(B_X),\ \|f\| \leq 1, \delta_y\circ C_{\mathbf 0}(f) = 0 \}\\
&= & \sup\{|\overline f(g(y))|: f\in\H^\infty(B_X),\ \|f\| \leq 1, f(0) = 0 \}= \|g(y)\|.
\end{eqnarray*}
Hence,  $\sigma( C_g,C_{\mathbf 0}) = \|g\|$.

(b) Since $g \in S_{\mathcal{H}^\infty(B_Y,X^{**})}$ there exist sequences $(y_n)\subset B_Y$ and $(x_n^*)\subset S_{X^*}$ such that $g(y_n)(x_n^*)\to 1$. Let us take $\lambda_n= h(y_n)(x_n^*)$ and note that $|\lambda_n|\le \|h\|<1$, for every $n$. We consider, for each $n$ and $m$, the following function defined on $B_X$:
\[f_{n,m}(\cdot) = \frac{(x^*_n(\cdot))^m - \lambda_n^m}{\| (x^*_n)^m - \lambda_n^m  \|}.\]
It is clear that $f_{n,m}$ belongs to $\H^\infty(B_X)$, $\Psi(f_{n,m})(y_n)=0$ and $\|f_{n,m}\|=1$. Thus,
\begin{eqnarray*}
\sigma(\Phi, \Psi) & = & \sup_{y\in B_Y} \sup\{ |\Phi(f)(y)|:\ f\in \H^\infty(B_X),\, \|f\|\le 1,\, \Psi(f)(y)=0\}\\
&\ge & \sup_{n,m} |\Phi(f_{n,m})(y_n)|=\sup_{n,m} \frac{|g(y_n)(x_n^*)^m-\lambda_n^m|}{\| (x^*_n)^m - \lambda_n^m  \|}\\
&\ge & \sup_{n,m}\frac{|g(y_n)(x_n^*)|^m-\|h\|^m}{1+\|h\|^m}=1.
\end{eqnarray*}
Consequently,  $\Phi$ and $\Psi$ lie in different Gleason parts.

(c) For any $y\in B_Y$ we know that $\delta_y\circ\Phi$ and $\delta_y\circ\Psi$ are in the fibers (with respect to the scalar-valued spectrum $\M_\infty(B_X)$) over $g(y)\in B_{X^{**}}$ and $h(y)\in S_{X^{**}}$, respectively. By \cite[Proposition 1.1]{AronDimantLassalleMaestre}, $\rho(\delta_y\circ\Phi, \delta_y\circ\Psi)=1$ and hence $\sigma(\Phi, \Psi)=1$.
\end{proof}

The statement of item (a) of the previous proposition, for functions $g \in B_{\mathcal{H}^\infty(B_Y,X)}$, was proved in \cite[Proposition 3]{AronGalindoLindstrom}. Also, item (b), in the particular case where $\Phi$ and $\Psi$ are composition homomorphisms, appeared in \cite[Proposition 5]{AronGalindoLindstrom}.

Recall that in the previous section we separated the fibers over functions $g\in\overline B_{\H^\infty(B_Y,X^{**})}$ into two cases:
\begin{enumerate}
\item[(i)] $g(B_Y)\subset B_{X^{**}}$.
\item[(ii)] $g(B_Y)\subset S_{X^{**}}$.
\end{enumerate}
Now, in light of the above proposition, to study Gleason parts it is relevant to split the first condition to distinguish whether the norm of $g$ is either 1 or smaller. Hence, the possible fibers to consider (with no intersection of Gleason parts) are:
\begin{enumerate}[(i)]
\item Fibers over functions $g$ with  $\|g\|<1$. Referred to as {\em interior fibers}.
\item Fibers over functions $g$ with $g(B_Y)\subset B_{X^{**}}$ and $\|g\|=1$. Referred to as {\em middle fibers}.
\item Fibers over functions $g$ with $g(B_Y)\subset S_{X^{**}}$. Referred to as {\em edge fibers}.
\end{enumerate}

Note also that, from (a), we have  $\{ C_g:\ g\in B_{\mathcal{H}^\infty(B_Y,X^{**})}\}\subset \mathcal{GP}(C_{\mathbf 0})$. This inclusion could be strict, for instance, when there is a polynomial on $X$ which is not weakly continuous on bounded sets, as the following result shows. This is a vector-valued version of \cite[Proposition 1.2 and Corollary 1.3]{AronDimantLassalleMaestre}.

\begin{proposition}{\label{basic 2}}
Let $X$ and $Y$ be Banach spaces.
\begin{enumerate}[\upshape (a)]
\item Let $(g_\alpha)$ be a net in $B_{\mathcal{H}^\infty(B_Y,X^{**})}$ with $\|g_\alpha\|\le r<1$, for all $\alpha$. If the net of composition homomorphisms $(C_{g_\alpha})$ is weak-star convergent to an element $\Phi$ in $\M_\infty(B_X, B_Y)$ then $\Phi$ is contained in the Gleason part $\mathcal{GP}(C_{\mathbf 0})$.
\item If there exists a polynomial on $X$ which is not weakly continuous on bounded sets, then $\{ C_g:\ g\in B_{\mathcal{H}^\infty(B_Y,X^{**})}\}$ is a proper subset of  $\mathcal{GP}(C_{\mathbf 0})$.
\end{enumerate}
\end{proposition}

\begin{proof}
(a) Take any $f \in \H^\infty(B_X)$ such that $\|f\| = 1$ and $f(0) = 0$, and an element $y\in B_Y$.  By the weak-star convergence, for any fixed $\varepsilon>0$ such that $r + \varepsilon <1$ we can find $\alpha$ such that $|C_{g_\alpha} (f)(y) - \Phi(f)(y)| < \varepsilon.$
Then, \[|\Phi(f)(y) - C_{\mathbf 0}(f)(y)| \leq \varepsilon + |C_{\mathbf 0}(f)(y) - \Phi_{g_\alpha} (f)(y)|\le \varepsilon + \sigma(\Phi_{g_\alpha},C_{\mathbf 0})=\varepsilon + \|g_\alpha\| < \varepsilon + r.\] Thus, $\sigma(\Phi,C_{\mathbf 0}) < 1,$ which concludes the proof.

(b) Working as in Theorem \ref{Inyectar disco Minf} (which, in turn, refers to Theorem \ref{inyectar disco}) we can construct a net $(C_{g_\alpha})$, as in item (a), that is  weak-star convergent to a homomorphism $\Phi$ which is in the Gleason part of $C_{\mathbf 0}$ but it is not of composition type.
\end{proof}

Observe that the vector-valued spectrum $\M_\infty(B_X,B_Y)$ is a metric space when viewed as a subset of $\mathcal L(\H^\infty(B_X), \H^\infty(B_Y))$. Since its metric is given by $\|\Phi-\Psi\|$ we refer to it as the {\em Gleason metric}. The following proposition, which is a version of \cite[Proposition 1.6]{AronDimantLassalleMaestre}, gives conditions under which there is an isometry in the spectrum that maps each fiber onto another fiber.

\begin{proposition}
	Let $X,Y$ be Banach spaces and $\theta: B_X \to B_X$ be an automorphism. Then the mapping
	\begin{align*}
		\Lambda_{\theta}:\M_\infty(B_X,B_Y) &\to \M_\infty(B_X,B_Y)\\
		\Phi &\mapsto (f \mapsto \Phi(f\circ \theta))
	\end{align*}
	is an isometry with respect to the Gleason metric. Moreover, if $X$ is symmetrically regular and for every $x^* \in X^*$ both $x^*\circ \theta$ and $x^*\circ\theta^{-1}$ are uniform limits of finite type polynomials,  then for every $g \in \overline{B}_{\H^\infty(B_Y,X^{**})}$ we have that $\Lambda_{\theta}(\mathscr{F}(g)) = \mathscr{F}(\widetilde{\theta}\circ g).$
		\end{proposition}
		
		\begin{proof}
			For $\Phi$ and $\Psi$ in $\M_\infty(B_X,B_Y)$ we have that
				\begin{equation*}
					\| \Lambda_{\theta}(\Phi) - \Lambda_{\theta}(\Psi)\| = \sup_{\substack {f\in 		\H^\infty(B_X)\\\|f\| \le 1}} \|\Phi(f\circ \theta) - \Psi(f \circ \theta) \| \leq \|\Phi - \Psi\|.
				\end{equation*}
			Applying the same inequality to $\Lambda_{\theta^{-1}}$ and noting that $\Lambda_{\theta^{-1}} \circ \Lambda_{\theta} = Id$, we obtain the desired isometry.
			
			Assume now that $X$ is symmetrically regular and for every $x^* \in X^*$ we have that both $x^* \circ \theta$ and $x^* \circ \theta^{-1}$ lie in the closure of finite type polynomials. Fix $g \in \overline{B}_{\H^\infty(B_Y,X^{**})}$. For any $\Phi \in \mathscr{F}(g)$, $y \in B_Y$, $x^* \in X^*$ we have
				\begin{align*}
					\Lambda_{\theta}(\Phi)(x^*)(y) &= \Phi(x^*\circ \theta)(y).\\
					\intertext{Since $x^*\circ \theta$ is a uniform limit of finite type polynomials the same happens to $\widetilde{x^*\circ \theta}$, which implies that there is a unique extension of $\widetilde\theta$ to $\overline{B}_{X^{**}}$ through weak-star continuity. Thus, we can compute}
					\Phi(x^*\circ \theta)(y) &= \widetilde{x^* \circ \theta} (g(y))=\widetilde{\theta}(g(y))(x^*).		
				\end{align*}
			This means that  $\Lambda_{\theta}(\mathscr{F}(g))$ is contained in $\mathscr{F}(\widetilde{\theta} \circ g)$. Also, since $X$ is symmetrically regular, arguing as in the proof of \cite[Corollary 2.2]{Choi} we can see that
$\widetilde{\theta^{-1}}\circ \widetilde\theta=Id$,
 and so repeating the same argument as above for $\theta^{-1}$ instead of $\theta$ we obtain that
\[
\mathscr{F}(\widetilde{\theta} \circ g)=\Lambda_{\theta}\left[ \Lambda_{\theta^{-1}}(\mathscr{F}(\widetilde \theta\circ g))\right]\subset \Lambda_{\theta}(\mathscr{F}(\widetilde{\theta^{-1}}\circ\widetilde \theta\circ g))=\Lambda_{\theta}(\mathscr{F}( g)).
\]
Hence,  the desired equality between the fibers is proved.
		\end{proof}

Examples of automorphisms of the ball satisfying the conditions of the previous proposition are shown in \cite[Examples 1.7 and 1.8]{AronDimantLassalleMaestre}  for  $X=c_0$ and $X=\ell_2$. The conclusion about the fibers of the vector-valued spectrum then holds in these cases for any Banach space $Y$.

For the scalar-valued spectrum of a uniform algebra it is known that the image of an open convex set through an analytic injection is contained in a single Gleason part (see, for instance, \cite[Lemma 2.1]{Hoffman} or \cite[Proposition 3.4]{AronDimantLassalleMaestre}). By Lemma \ref{holomorphic mapping} the same result is valid for the vector-valued spectrum.

\begin{proposition}
Given Banach spaces $X$, $Y$ and $Z$ and an open convex set $U\subset Z$, let $F:U\to \M_\infty(B_X,B_Y)$ be an analytic injection. Then, $F(U)$ is contained in a single Gleason part.
\end{proposition}

This proposition combined with some of the results of Section \ref{Section-Minf} provides examples of situations where a ball is contained in the intersection of a fiber and a Gleason part. More precisely we have:

\begin{itemize}
\item If $x_0^{**}\in S_{X^{**}}$ is a norm attaining element and $g(y)=x_0^{**}$, for all $y\in B_Y$, consider the analytic injection $\Psi$ of Proposition \ref{FibrasS_X}. Then, $\Psi(B_{\H^\infty(B_Y)})$ is contained in the intersection of a Gleason part and $\mathscr{F}(g)$.
\item If there exists a polynomial on $X$ which is not weakly continuous on bounded sets, by Theorem \ref{Inyectar disco Minf}, for each $g \in B_{\H^\infty(B_Y,X^{**})}$ there is a copy of the complex disk $\mathbb{D}$ in the intersection of a Gleason part and $\mathscr{F}(g)$.
\item If there exists a polynomial on $X$ which is not weakly continuous on bounded sets, by Theorem \ref{Inyectar bola Minf}, for each constant function $g \in B_{\H^\infty(B_Y,X^{**})}$ there is a copy of the unit ball $B_{\H^\infty(B_Y)}$ in the intersection of a Gleason part and $\mathscr{F}(g)$.
\end{itemize}

The above examples and Proposition \ref{basic 1} (a) show situations in which Gleason parts contain balls; so we can say that they are {\em large} Gleason parts. Yet, in this spectrum, there also exist singleton Gleason parts. On the one hand, it is easily seen that any singleton Gleason part of the scalar-valued spectrum $\M_\infty(B_X)$ is also a singleton Gleason part of $\M_\infty(B_X, B_Y)$. On the other hand, singleton Gleason parts which are not in fibers over constant functions surely exist in $\M_\infty(B_X, B_X)$. For instance, the identity mapping is a singleton Gleason part. This was addressed in \cite{chu-hugli-mackey} showing that no composition operator might be in the same Gleason part, answering a conjecture stated in  \cite{AronGalindoLindstrom}, where a particular case was studied. A complete proof of the statement can be found in \cite{GalindoGamelinLindstrom}. Also, if $g:B_Y\to B_X$ is biholomorphic then the Gleason part containing $ C_g\in \M_\infty(B_X, B_Y)$ is a singleton. Indeed, it is readily seen that the mapping from $\mathcal{M_\infty}(B_Y) \to \mathcal{M}_\infty(B_X)$ given by $[\varphi \mapsto \varphi \circ C_g]$ maps strong boundary points to strong boundary points. The result then follows from \cite[Theorem 6.2]{GalindoGamelinLindstrom}.

\begin{example} \textbf{Relationship between fibers and Gleason parts.}\\
\rm The case $B_X=\D$ allows us to show how the relationship between fibers and Gleason parts is different whether we consider interior, middle or edge fibers.
We take into account the description of the fibers of the spectrum $\M_\infty(\D, B_Y)$ made in Theorem \ref{propDisk} along with what we know from Proposition \ref{basic 1} about Gleason parts. \\
\begin{itemize}
	\item Interior fibers of $\M_\infty(\D, B_Y)$ only contain the corresponding composition homomorphism. Then,  $\mathcal{GP}(C_{\mathbf 0})=\{ C_g:\ g\in B_{\H^\infty(B_Y)}\}$, so there is only one Gleason part through all the interior fibers.
	\item For edge fibers, take $\lambda\not=\mu$ with $|\lambda|=|\mu|=1$ and $\Phi\in\F(g)$, $\Psi\in\F(h)$, where $g(y)= \lambda$ and $h(y)= \mu$, for all $y\in B_Y$. Then, for any $y$, we have that $\delta_y\circ \Phi$ and $\delta_y\circ \Psi$ are homomorphisms in the scalar-valued spectrum $\M_\infty(\D)$ belonging to the fibers over $\lambda$ and $\mu$, respectively. Thus, it is known  that they belong to different Gleason parts. Hence, $1=\rho(\delta_y\circ \Phi, \delta_y\circ \Psi)\le \sigma(\Phi, \Psi)$ and so $\mathcal{GP}(\Phi)\not= \mathcal{GP}(\Psi)$.  Therefore,  no Gleason part could have elements from different edge fibers.
\end{itemize}
	The transition between one Gleason part containing all the fibers (interior case) and Gleason parts inside the fibers (edge case) is made by the middle fibers:
\begin{itemize}
	\item Any middle fiber only contains the corresponding composition homomorphism, yet several (but not all) middle fibers might belong to the same Gleason part.
\end{itemize}

We show this last situation with the following example, which is partially adapted from \cite[Example 2]{MacCluerOhnoZhao}.
For $y^*\in S_{X^*}$, consider the functions $g(y)=y^*(y)$, $h(y)=\frac{y^*(y)+1}{2}$ and $i(y)=\frac{y^*(y)+1}{2}+k(y^*(y)-1)^2$, for every $y\in B_Y$, where $0<k<\frac18$. They are all {\em middle} functions in $\H^\infty(B_Y)$ and the composition homomorphisms associated to them satisfy  $\mathcal{GP}(C_g)\not= \mathcal{GP}(C_h)=\mathcal{GP}(C_i)$. Indeed, take a sequence $(y_n)$ in $B_Y$ such that $y^*(y_n)\to -1$. Then
\[
\rho(\delta_{y_n}\circ C_g, \delta_{y_n}\circ C_h)= \rho(\delta_{g(y_n)},\delta_{h(y_n)})=\left|\frac{g(y_n)-h(y_n)}{1-\overline {g(y_n)} h(y_n)}\right| \to 1.
\] This implies that $\mathcal{GP}(C_g)\not= \mathcal{GP}(C_h)$. On the other hand we have
\begin{eqnarray*}
\sigma(C_h, C_i) & =&\sup_{y\in B_Y}\rho(\delta_{y}\circ C_h, \delta_{y}\circ C_i)= \sup_{y\in B_Y} \rho(\delta_{h(y)},\delta_{i(y)})\\
&=& \sup_{y\in B_Y} \left|\frac{h(y)-i(y)}{1-\overline {h(y)} i(y)}\right| = \sup_{y\in B_Y} \left|\frac{k(y^*(y)-1)^2}{1- \frac{\overline{y^*(y)}+1}{2} (\frac{y^*(y)+1}{2}+k(y^*(y)-1)^2)}\right|\\
&=& \sup_{z\in\D} \left|\frac{k(z-1)^2}{1- \frac{\overline{z}+1}{2} (\frac{z+1}{2}+k(z-1)^2)}\right| <1,
\end{eqnarray*} where the last inequality is proved in \cite[Example 2]{MacCluerOhnoZhao}.
Therefore, $\mathcal{GP}(C_h)= \mathcal{GP}(C_i)$.
\end{example}

\textbf{Acknowledgements.} We would like to thank Daniel Carando for helpful conversations.

\providecommand{\WileyBibTextsc}{}
\let\textsc\WileyBibTextsc
\providecommand{\othercit}{}
\providecommand{\jr}[1]{#1}
\providecommand{\etal}{~et~al.}


\begin{thebibliography}{[10]}

\bibitem{AronBerner}
\textsc{R.\,M. Aron} and  \textsc{P.\,D. Berner},
A Hahn-Banach extension theorem for analytic mappings,
\jr{Bull. Soc. Math. France} \textbf{106}, 3--24
(1978).

\bibitem{AronColeGamelin}
\textsc{R.\,M. Aron},  \textsc{B.\,J. Cole},  and  \textsc{T.\,W. Gamelin},
Spectra of algebras of analytic functions on a {B}anach space,
\jr{J. Reine Angew. Math.} \textbf{415}, 51--93 (1991).

\bibitem{AronDimant}
\textsc{R.\,M. Aron} and  \textsc{V.~Dimant},
Sets of weak sequential continuity for polynomials,
\jr{Indag. Math. (N.S.)} \textbf{13}(3), 287--299 (2002).

\bibitem{AronDimantLassalleMaestre}
\textsc{R.~M.~Aron},  \textsc{V.~Dimant},  \textsc{S.~Lassalle},  and
\textsc{M.~Maestre},
Gleason parts for algebras of holomorphic functions in infinite dimensions,
\jr{preprint arXiv:1902.01735} (2019).

\bibitem{AronFalcoGarciaMaestre}
\textsc{R.\,M. Aron},  \textsc{J.~Falc\'{o}},  \textsc{D.~Garc\'{\i}a},  and
\textsc{M.~Maestre},
Analytic structure in fibers,
\jr{Studia Math.} \textbf{240}(2), 101--121 (2018).

\bibitem{AronGalindoGarciaMaestre}
\textsc{R.\,M. Aron},  \textsc{P.~Galindo},  \textsc{D.~Garc\'{\i}a},  and
\textsc{M.~Maestre},
Regularity and algebras of analytic functions in infinite dimensions,
\jr{Trans. Amer. Math. Soc.} \textbf{348}(2), 543--559 (1996).

\bibitem{AronGalindoLindstrom}
\textsc{R.~Aron},  \textsc{P.~Galindo},  and  \textsc{M.~Lindstr\"{o}m},
Connected components in the space of composition operators in {$H^\infty$}
functions of many variables,
\jr{Integral Equations Operator Theory} \textbf{45}(1), 1--14 (2003).

\othercit
\bibitem{Bear}
\textsc{H.\,S. Bear},
Lectures on {G}leason parts, Lecture Notes in Mathematics, Vol. 121
(Springer-Verlag, Berlin-New York, 1970).

\bibitem{BoydRyan}
\textsc{C.~Boyd} and  \textsc{R.\,A. Ryan},
Bounded weak continuity of homogeneous polynomials at the origin,
\jr{Arch. Math. (Basel)} \textbf{71}(3), 211--218 (1998).

\othercit
\bibitem{carlsson}
\textsc{L.~Carlsson},
Ideals and boundaries in Algebras of Holomorphic functions,
Doctoral dissertation, Ume\r{a} universitet, 2006.

\othercit
\bibitem{Chae}
\textsc{S.\,B. Chae},
Holomorphy and calculus in normed spaces, Monographs and Textbooks in Pure and
Applied Mathematics,  Vol.\,92 (Marcel Dekker, Inc., New York, 1985),
With an appendix by Angus E. Taylor.

\bibitem{Choi}
\textsc{Y.\,S. Choi},  \textsc{D.~Garc\'{\i}a},  \textsc{S.\,G. Kim},  and
\textsc{M.~Maestre},
Composition, numerical range and {A}ron-{B}erner extension,
\jr{Math. Scand.} \textbf{103}(1), 97--110 (2008).

\bibitem{chu-hugli-mackey}
\textsc{C.\,H. Chu},  \textsc{R.\,V. H\"{u}gli},  and  \textsc{M.~Mackey},
The identity is isolated among composition operators,
\jr{Proc. Amer. Math. Soc.} \textbf{132}(11), 3305--3308 (2004).

\bibitem{ColeGamelinJohnson}
\textsc{B.\,J. Cole},  \textsc{T.\,W. Gamelin},  and  \textsc{W.\,B. Johnson},
Analytic disks in fibers over the unit ball of a {B}anach space,
\jr{Michigan Math. J.} \textbf{39}(3), 551--569 (1992).

\bibitem{DavieGamelin}
\textsc{A.\,M. Davie} and  \textsc{T.\,W. Gamelin},
A theorem on polynomial-star approximation,
\jr{Proc. Amer. Math. Soc.} \textbf{106}(2), 351--356 (1989).

\bibitem{Deghoul}
\textsc{D.~Deghoul},
Construction de caract\`eres exceptionnels sur une alg\`ebre de {F}r\'{e}chet,
\jr{C. R. Acad. Sci. Paris S\'{e}r. I Math.} \textbf{312}(8), 579--580 (1991).

\bibitem{DiGaMaSe}
\textsc{V.~Dimant},  \textsc{D.~Garc\'{\i}a},  \textsc{M.~Maestre},  and
\textsc{P.~Sevilla-Peris},
Homomorphisms between algebras of holomorphic functions,
\jr{Abstr. Appl. Anal.} pp.\,Art. ID 612304, 12 (2014).

\bibitem{DimantSinger2}
\textsc{V.~Dimant} and  \textsc{J.~Singer},
Homomorphisms between algebras of holomorphic functions on the infinite
polydisk,
\jr{Work in progress}.

\othercit
\bibitem{Dineen}
\textsc{S.~Dineen},
Complex analysis on infinite-dimensional spaces, Springer Monographs in
Mathematics (Springer-Verlag London, Ltd., London, 1999).

\bibitem{Farmer}
\textsc{J.\,D. Farmer},
Fibers over the sphere of a uniformly convex {B}anach space,
\jr{Michigan Math. J.} \textbf{45}(2), 211--226 (1998).

\bibitem{GalindoGamelinLindstrom}
\textsc{P.~Galindo},  \textsc{T.\,W. Gamelin},  and
\textsc{M.~Lindstr\"{o}m},
Composition operators on uniform algebras, essential norms, and hyperbolically
bounded sets,
\jr{Trans. Amer. Math. Soc.} \textbf{359}(5), 2109--2121 (2007).

\bibitem{GGL}
\textsc{P.~Galindo},  \textsc{T.\,W. Gamelin},  and
\textsc{M.~Lindstr\"{o}m},
Fredholm composition operators on algebras of analytic functions on {B}anach
spaces,
\jr{J. Funct. Anal.} \textbf{258}(5), 1504--1512 (2010).

\bibitem{GalindoLindstrom}
\textsc{P.~Galindo} and  \textsc{M.~Lindstr\"{o}m},
Gleason parts and weakly compact homomorphisms between uniform {B}anach
algebras,
\jr{Monatsh. Math.} \textbf{128}(2), 89--97 (1999).

\othercit
\bibitem{Gamelin}
\textsc{T.\,W. Gamelin},
Homomorphisms of uniform algebras,
in: Recent progress in functional analysis ({V}alencia, 2000),  North-Holland
Math. Stud.,  Vol.\,189 (North-Holland, Amsterdam, 2001),  pp.\,95--105.

\bibitem{Gorkin}
\textsc{P.~Gorkin},
Gleason parts and {COP},
\jr{J. Funct. Anal.} \textbf{83}(1), 44--49 (1989).

\bibitem{GorkinMortini}
\textsc{P.~Gorkin} and  \textsc{R.~Mortini},
Norms and essential norms of linear combinations of endomorphisms,
\jr{Trans. Amer. Math. Soc.} \textbf{358}(2), 553--571 (2006).

\othercit
\bibitem{GorkinMortiniSuarez}
\textsc{P.~Gorkin},  \textsc{R.~Mortini},  and  \textsc{D.~Su\'{a}rez},
Homotopic composition operators on {$H^\infty(B^n)$},
in: Function spaces ({E}dwardsville, {IL}, 2002),  Contemp. Math.,  Vol.\,328
(Amer. Math. Soc., Providence, RI, 2003),  pp.\,177--188.

\othercit
\bibitem{lemmers}
\textsc{O.~Lemmers},
On the Gleason problem,
Doctoral dissertation, Universiteit van Amsterdam, 2002.

\bibitem{Hoffman}
\textsc{K.~Hoffman},
Bounded analytic functions and {G}leason parts,
\jr{Ann. of Math. (2)} \textbf{86}, 74--111 (1967).

\bibitem{HozokawaIzuchiZheng}
\textsc{T.~Hosokawa},  \textsc{K.~Izuchi},  and  \textsc{D.~Zheng},
Isolated points and essential components of composition operators on
{$H^\infty$},
\jr{Proc. Amer. Math. Soc.} \textbf{130}(6), 1765--1773 (2002).

\bibitem{MacCluerOhnoZhao}
\textsc{B.~MacCluer},  \textsc{S.~Ohno},  and  \textsc{R.~Zhao},
Topological structure of the space of composition operators on {$H^\infty$},
\jr{Integral Equations Operator Theory} \textbf{40}(4), 481--494 (2001).

\othercit
\bibitem{mortini}
\textsc{R.~Mortini},
Gleason parts and prime ideals in $H^\infty$,
in: Linear and complex analysis. {P}roblem book 3, edited by V.\,P. Havin and
N.\,K. Nikolski, Lecture Notes in Mathematics Vol.\,1574 (Springer-Verlag,
Berlin, 1994),  pp.\,136--138.

\othercit
\bibitem{MujicaLibro}
\textsc{J.~Mujica},
Complex analysis in {B}anach spaces, North-Holland Mathematics Studies,
Vol.\,120 (North-Holland Publishing Co., Amsterdam, 1986),
Holomorphic functions and domains of holomorphy in finite and infinite
dimensions, Notas de Matem\'{a}tica [Mathematical Notes], 107.

\bibitem{Mujica}
\textsc{J.~Mujica},
Linearization of bounded holomorphic mappings on {B}anach spaces,
\jr{Trans. Amer. Math. Soc.} \textbf{324}(2), 867--887 (1991).

\othercit
\bibitem{Rudin-FunctionTheoryCn}
\textsc{W.~Rudin},
Function theory in the unit ball of {$\Bbb C^n$}, Classics in Mathematics
(Springer-Verlag, Berlin, 2008),
Reprint of the 1980 edition.

\bibitem{Suarez}
\textsc{D.~Su\'{a}rez},
Maximal {G}leason parts for {$H^\infty$},
\jr{Michigan Math. J.} \textbf{45}(1), 55--72 (1998).


\end{thebibliography}
\end{document}